\documentclass{amsart}

\usepackage{amssymb}
\usepackage[a4paper, margin={1in}]{geometry}
\usepackage{mathtools}
\usepackage{leftindex}
\usepackage{xcolor}
\usepackage{biblatex}
\usepackage{hyperref}
\usepackage{cleveref}

\newtheorem{theorem}{Theorem}[section]

\newtheorem{proposition}[theorem]{Proposition}

\theoremstyle{definition}
\newtheorem{definition}[theorem]{Definition}

\theoremstyle{remark}

\numberwithin{equation}{section}

\newcommand{\BC}{{\mathbb{C}}}

\newcommand{\BH}{{\mathbb{H}}}

\newcommand{\BN}{{\mathbb{N}}}

\newcommand{\BQ}{{\mathbb{Q}}}
\newcommand{\BR}{{\mathbb{R}}}

\newcommand{\BZ}{{\mathbb{Z}}}

\newcommand{\Fe}{{\mathfrak{e}}}

\newcommand{\CD}{{\mathcal D}}

\newcommand{\CF}{{\mathcal F}}

\newcommand{\CM}{{\mathcal M}}

\newcommand{\CO}{{\mathcal O}}

\newcommand{\CW}{{\mathcal W}}

\newcommand{\set}[1]{\left\lbrace #1 \right\rbrace}		
\newcommand{\map}[3]{#1 : #2 \rightarrow #3} 

\DeclareMathOperator{\Gr}{Gr}
\DeclareMathOperator{\SL}{SL}
\DeclareMathOperator{\Mp}{Mp}
\DeclareMathOperator{\GL}{GL}
\newcommand{\Spin}{\operatorname{Spin}}
\newcommand{\bs}{\backslash}
\newcommand{\diff}[1]{\partial_{#1}}
\newcommand{\invdiff}[1]{\diff{#1}^{-1}}

\DeclareMathOperator{\tr}{tr}

\renewcommand{\O}{\operatorname{O}}

\newcommand{\abs}[1]{\left\lvert #1 \right\rvert}
\newcommand{\regint}{\int^{\text{reg}}}
\newcommand{\siegTheta}[2]{\Theta^{(#1)}_{#2}}
\newcommand{\siegtheta}[2]{\theta^{(#1)}_{#2}}
\DeclareMathOperator{\sgn}{sgn}

\newcommand{\groupring}[1]{\BC[#1'/#1]}

\newcommand{\bpsm}{\begin{psmallmatrix}}
\newcommand{\epsm}{\end{psmallmatrix}}

\begin{document}

\title[Fourier coefficients of meromorphic modular forms]{On the divisibility properties of the Fourier coefficients of meromorphic Hilbert modular forms}

\author{Baptiste Depouilly}
\address{}
\curraddr{}
\email{}
\thanks{}

\subjclass[]{}

\date{}

\dedicatory{}

\begin{abstract}
    Following Zagier, this work studies the rationality and divisibility of Fourier coefficients of meromorphic Hilbert modular forms associated with real quadratic fields, using theta lifts and weak Maass forms. We establish conditions where these coefficients are rational with bounded denominators and demonstrate divisibility properties under suitable linear combinations.
\end{abstract}

\maketitle



\section{Introduction}

\subsection{Fourier coefficients of meromorphic Hilbert modular forms}

Fix a real quadratic field $F$ of discriminant $D$, denote by $\invdiff{F}$ its inverse different, and fix an even integer $k \geq 4$. In \cite{zagier1975}, Zagier defined a family of Hilbert modular forms
\begin{equation}
    \omega_m(z_1, z_2) = \frac{(k-1)!}{2(2\pi)^k} \sum_{\substack{a.b \in \BZ \\ \nu \in \invdiff{F} \\ N(\nu) - ab = m}} (b z_1z_2 + \nu z_1 + \nu'z_2 + a )^{-k}
\end{equation}
for integers $m \in \frac{1}{D}\BZ_{\geq 0}$, which transform with parallel weight $k$ with respect to the Hilbert modular group $\SL_2(\CO_F)$.
The same construction can be realized taking $m < 0$ to produce a meromorphic Hilbert modular form which vanishes at the cusps and has poles at the Hirzebruch-Zagier divisors of $\SL_2(\CO_F) \bs (\BH \times \BH)$ of discriminant $m$. 
Readers might notice that our normalization differs from that of Zagier; the role of the normalizing constant will become clearer once we state our result.
In a Weyl chamber denoted $\CW$, the $\omega_m$'s admit a Fourier expansion of the form 
\[\omega_m(z_1, z_2) = \sum_{\nu \in \invdiff{F}} c_{\nu} e(\tr(\nu z)), \]
where $e(z) := e^{2\pi i z}$ and $\tr(\nu z) := \nu z_1 + \nu' z_2$. In this article, we exhibit rationality and divisibility properties for the Fourier coefficients $c_\nu$ when $m < 0$.

Let us set up the notation for our result. Denote by $\rho_L$ the Weil representation attached to $L:=\CO_F$ viewed as a rank $2$ $\BZ$-module, and $\overline{\rho}_L$ its dual representation. 
For $\nu \in \invdiff{F}$, let $\nu_0$ be the unique element of $\invdiff{F}$ such that $\BQ \nu \cap \invdiff{F} = \BZ \nu_0$. We call $\nu_0$ the primitive part of $\nu$ and denote by $\ell_\nu$ the integer such that $\nu = \ell_\nu \nu_0$.
In this setting, we will leverage the fact that for a fixed negative $m \in \frac{1}{D}\BZ_{< 0}$, $\omega_m$ can be realized as the Doi-Naganuma theta lift of a vector-valued Maass-Poincaré series to prove the following result.
\begin{theorem}\label{thm:maintheorem} 
    For $m < 0$ and an even integer $k \geq 4$, if the space of cusp forms $S_{k,\overline{\rho}_L}$ is trivial, then, up to multiplication by an integer that doesn't depend on $\nu$,
    the $\nu$-th coefficient of $\omega_m$ is divisible by $(D\ell_\nu N(\nu_0))^{k-1}$ for all $\nu \in \invdiff{F}$.
\end{theorem}

We obtain a similar result when the space $S_{k, \overline{\rho}_L}$ is non-trivial by taking linear combinations of the $\omega_m$'s. The precise statement is given in \Cref{theo:mainTheorem}.

\subsection{Numerical examples} We list some coefficients of $\omega_{m}$ for $F = \BQ(\sqrt{5})$, $m=-1/5$, and weight 4.
We observe the following:
\vspace{5pt}

\begin{center}
  \renewcommand{\arraystretch}{1.8}
  \begin{tabular}{|| c | c | c | c | c | c ||}
    \hline
    $\nu$ & $\ell_\nu$ & $(D \ell_\nu N(\nu_0))^{k-1}$ & $N(\nu)$ & $c_{\nu}$ & $c_{\nu}$ factored \\
    \hline \hline
    2 & 2 & $ 2^3 \cdot 5^3$ & $2^2$ & $3050000$ & $2^4 \cdot 5^5 \cdot 61$\\
    \hline
    $\frac{-2+2\sqrt{5}}{\sqrt{5}}$ & $4$ & $2^6$ & $2^4 / 5$ & $201728$ & $2^6 \cdot 3 \cdot 1051$ \\
    \hline
    3 & $3$ & $3^3 \cdot 5^3$ & 9 & $2649037500$ & $2^2 \cdot 3^3 \cdot 5^5 \cdot 47 \cdot 167$ \\
    \hline
    $\frac{-3+3\sqrt{5}}{2\sqrt{5}}$ & $3$ & $3^3$ & $9/5$ & $-18198$ & $-1 \cdot 2 \cdot 3^3 \cdot 337$ \\
    \hline
    $\frac{-2}{\sqrt{5}} + 3$ & $1$ & $41$ & $41/5$ & $334266850$ & $2 \cdot 5^2 \cdot 41^3 \cdot 97$ \\
    \hline
    $\frac{-1}{2\sqrt{5}} + \frac{5}{2}$ & $1$ & $31^3$ & $31/5$ & $27675839$ & $31^3 \cdot 929$ \\
    \hline
    $\frac{5+7\sqrt{5}}{2\sqrt{5}}$ & $1$ & $(5 \cdot 11)^{3}$ & $11$ & $20708696250$ & $2 \cdot 3^3 \cdot 5^4 \cdot 11^3 \cdot 461$ \\
    \hline 
    $\frac{5+5\sqrt{5}}{2\sqrt{5}}$ & $5$ &  $5^3$ & $5$ & $31562625$ & $3 \cdot 5^3 \cdot 17 \cdot 4951$ \\
    \hline
  \end{tabular}
\end{center}
\vspace{5pt}
These were numerically computed using Sage and the formula given in \Cref{theo:fourierExpOmega}.

\subsection{How this article is organized} In \Cref{sec:HarmWeakForms}, we set up the notation and recall important results in the theory of harmonic Maass forms.
In \Cref{sec:OrthModForms}, we introduce orthogonal modular forms and recall their connection with Hilbert modular forms. We additionally introduce our object of study, the functions $\omega_{\beta, m}(Z)$, we detail the location of their poles and their Fourier expansion.
In \Cref{sec:DoiNagLift}, we define the vector-valued Doi-Naganuma theta lift and compute the image of the non-holomorphic Maass-Poincaré series through this lift.
Finally, in \Cref{sec:IntAndDiv}, we study the Fourier coefficients of $\omega_{\beta, m}(Z)$ and prove our main result.


\section{Harmonic weak Maass forms for the Weil representation} \label{sec:HarmWeakForms}

We fix here some of the notation we will use in this paper. We will have $D$ be a positive fundamental discriminant and $F$ be the unique (up to isomorphism) real quadratic field of discriminant $D$.
Additionally, $\CO_F$ will denote the ring of integers of $F$ and $\CO_F^*$ is its group of units. We will use the symbol $\invdiff{F}$ for the inverse different of $F$.
If $x \in F$ is a number field element, $x'$ will denote its Galois conjugate, $N(x)$ its norm, and $\tr(x)$ its trace.

\subsection{The Weil representation}

Let $\Mp_2(\BR)$ denote the double cover of $\SL_2(\BR)$ given by pairs $(M, \phi)$, where 
$M = \bpsm a & b \\ c & d \epsm \in \SL_2(\BR)$ and $\map{\phi}{\BH}{\BC}$ is a holomorphic square-root of $c\tau + d$.
The subgroup $\Mp_2(\BZ)$ is generated by the two elements $S=\left(\bpsm 0 & -1 \\ 1 & 0 \epsm, \sqrt{\tau}\right)$ and $T = \left(\bpsm 1 & 1 \\ 0 & 1 \epsm, 1\right)$.
Now, let $L$ be an even lattice endowed with a quadratic form $Q$ of signature $(b^+, b^-)$. We denote by $(\cdot, \cdot)$ the bilinear form associated to $Q$ and by $L'$ the dual lattice of $L$. 
The quotient $L'/L$ is a finite abelian group called the discriminant group of $L$. Moreover, consider the group ring $\groupring{L}$ with its standard basis
$\set{\Fe_\beta}_{\beta \in L'/L}$ and the hermitian inner product $\langle \Fe_\beta, \Fe_\gamma \rangle = \delta_{\beta, \gamma}$.

There is a unique unitary representation $\map{\rho_L}{\Mp_2(\BZ)}{\GL(\groupring{L})}$ called the Weil representation defined by 
\begin{align*}
    \rho_L(S)\Fe_\beta & = \frac{\sqrt{i}^{b^- - b^+}}{\sqrt{\abs{L'/L}}}\sum_{\gamma \in L'/L} \Fe_\gamma(-(\beta, \gamma)) \ \text{and} \ \rho_L(T) \Fe_\beta = \Fe_\beta(Q(\beta)),
\end{align*}
where we used the notation $\Fe_\beta(x) = e(x)\Fe_\beta = e^{2\pi i x}\Fe_\beta$.
If $L^-$ denotes the pair $(L, -Q)$, then the corresponding Weil representation $\rho_{L^-}$ is the dual representation $\overline{\rho}_L$ to $\rho_{L}$.

\subsection{Vector-valued harmonic Maass forms}

Let $k \in \frac{1}{2}\BZ$. For functions $f: \BH \rightarrow \groupring{L}$ and $(M, \phi) \in \Mp_2(\BZ)$, we define the 
slash operator of weight $k$ by 
\[(f \vert_k (M,\phi)) (\tau) = \phi(\tau)^{-2k} \rho_L^{-1}(M,\phi)f(M\tau).\]
Let $\Gamma$ be a subgroup of $\Mp_2(\BZ)$ of finite index. We define the space $H_{k, \rho_L}(\Gamma)$ of \emph{weak Maass forms} of weight $k$ with representation $\rho_L$ 
for the group $\Gamma$ similarly as Bruinier-Funke in \cite{bruinierfunke2004}.

Additionally, we recall the definitions of the Maass lowering and raising operators in weight $k$ on smooth functions $g: \BH \rightarrow \groupring{L}$.
They act componentwise and are given by 
\[L_k = -2iv^2 \frac{\partial}{\partial \overline{\tau}} \quad \text{and} \quad R_k = 2iv^{-k}\frac{\partial}{\partial \tau} v^k = 2i \frac{\partial}{\partial \tau} + \frac{k}{v}.\]
They lower, respectively raise, the weight of a form by 2. We can iterate these operators and use the notation 
\begin{align*}
    R_{k}^n & = R_{k + 2(n-1)} \circ \cdots \circ R_{k+2} \circ R_k.
\end{align*}

\subsection{Non-holomorphic Poincaré series} 
Let $k \in \frac{1}{2}\BZ$. Denote by  $M_{\mu,\nu}$ the usual $M$-Whittaker function. 
For fixed values $s \in \BC$ and $m \in \frac{1}{D}\BZ$, define
\begin{equation}
    \CM_{m,k}(v,s) = \left\{\begin{array}{ll}
        \Gamma(2s)^{-1}(4\pi\abs{m}v)^{-k/2}M_{\frac{k}{2}\sgn(m), s-\frac{1}{2}}(4\pi\abs{m}v), & m \neq 0, \\
        v^{s-\frac{k}{2}}, & m=0.
    \end{array}\right.
\end{equation}
The function 
\[\varphi_{m,k}(\tau,s) := \CM_{m,k}(v,s)e^{2\pi i mu}\]
is then an eigenfunction of the weight $k$ hyperbolic Laplacian operator. At $s = k/2$, we have
\begin{equation}
    \CM_{m,k}(v,k/2) = e^{-2\pi m v} \Gamma(k)^{-1}.
\end{equation}
Finally, we follow \cite{JeonKangKim} and use $\varphi_{m,k}(\tau,s)$ as a kernel function to define the Maass-Poincaré series:
\begin{definition}
    With the above notation we define the Poincaré series $F_{\beta, m}$ of index $(\beta,m)$ by 
    \begin{equation}
        F_{\beta, m}(\tau, s) = \frac{1}{2} \sum_{(M, \phi) \in \tilde{\Gamma}_\infty \bs \Mp_2(\BZ)} \left[\varphi_{m,k}(\tau,s) \Fe_\beta\right] \vert_k (M, \phi),
    \end{equation}
    where $\tau = u+iv \in \BH$ and $s = \sigma + it \in \BC$ with $\sigma > 1$.
\end{definition}

This Poincaré series transforms like a modular form of weight $k$ with respect to $\Mp_2(\BZ)$ and converges absolutely and uniformly on compacta for $\Re(s) > 1$. Moreover, $F_{\beta, m}(\tau,s)$
is an eigenfunction of the weight $k$ hyperbolic Laplacian.

The Fourier expansion of the Maass-Poincaré series has been computed by Bruinierin \cite{bruinier2002} in the case of negative weight $k$. We can also tackle the positive weight Maass-Poincaré
series by adjusting the definition of the $M$-Whittaker function, as is done in \cite{JeonKangKim}.

\begin{proposition}
    For $m \in \frac{1}{D}\BZ$, $\beta \in L'/L$, and $k \in \BZ$, $F_{\beta, m}(\tau,s)$ admits a Fourier expansion of the form 
    \[\CM_{m,k}(v,s)\left(\Fe_\beta(m u) + \Fe_{-\beta}(m u) \right) + \sum_{\gamma \in L'/L} \sum_{\substack{n \in \BZ + Q(\gamma)}} c_{m,k}(\gamma, n; s) \CW_{n,k}(v,s) \Fe_\gamma(n u).\]
    The exact formulae for the coefficients $c_{m,k}(\gamma, n; s)$ can be found in the previously cited references.
\end{proposition}

\section{Orthogonal modular forms}\label{sec:OrthModForms}
 
\subsection{\texorpdfstring{The Grassmannian of $\BH^2$}{The Grassmannian of H2}}

We consider the lattice
\[L = \left\{
    \begin{pmatrix}
        a & \nu' \\ \nu & b
    \end{pmatrix}: a,b \in \BZ, \ \nu \in \CO_F
    \right\}\]
and the associated vector spaces $V := L \otimes \BQ$ and $V(\BR) := L \otimes \BR$.
We endow them with the quadratic form 
\[Q(X) := -\det(X),\] which corresponds to the bilinear form 
\[(X,Y) = -\tr(X Y^*),\]
where $\bpsm a & b \\ c & d \epsm^* = \bpsm d & -b \\ -c & a \epsm$. The real quadratic space $(V(\BR), Q)$ has signature 
$(2,2)$. 

An analysis of the Clifford algebra of $V(\BR)$ yields $\Spin_{V(\BR)} \cong \SL_2(\BR)^2$, which acts isometrically on $V(\BR)$ by 
\[(g_1, g_2) \bullet X = g_1 X g_2^t.\]
In particular, $g \in \SL_2(F)$ acts isometrically on $V(\BR)$ via
\[(g, g') \bullet X = gX(g')^t,\]
where $g'$ denotes the entry-wise Galois conjugate of the matrix $g$.
To any complex number $z = x+iy \in \BH$, we associate the matrix 
\[\begin{pmatrix}
    \sqrt{y} & x/\sqrt{y} \\
    0 & 1/\sqrt{y}
\end{pmatrix} \in \SL_2(\BR),\]\
and for $Z := (z_1, z_2) \in \BH^2$, we let $(g_{z_1}, g_{z_2}) \in \SL_2(\BR)^2$ act on the orthogonal basis 
\[e_1 = \begin{pmatrix}
    -1 & 0 \\ 0 & 1
\end{pmatrix}, \ e_2 = \begin{pmatrix}
    0 & 1 \\ 1 & 0
\end{pmatrix}, \ e_3 = \begin{pmatrix}
    1 & 0 \\ 0 & 1
\end{pmatrix}, \ e_4 = \begin{pmatrix}
    0 & -1 \\ 1 & 0
\end{pmatrix}\]
of $V(\BR)$. This yields a family of orthonormal bases
\[\begin{array}{cc}
    X_1(Z) = \displaystyle \frac{1}{\sqrt{y_1 y_2}} \begin{pmatrix}
        \Re(z_1 z_2) & \Re(z_1) \\ \Re(z_2) & 1
    \end{pmatrix}, & 
    X_2(Z) = \displaystyle \frac{1}{\sqrt{y_1 y_2}} \begin{pmatrix}
        \Im(z_1 z_2) & \Im(z_1) \\ \Im(z_2) & 0
    \end{pmatrix} \\
    X_3(Z) = \displaystyle \frac{1}{\sqrt{y_1 y_2}} \begin{pmatrix}
        \Re(\overline{z_1}z_2) & \Re(z_1) \\ \Re(z_2) & 1
    \end{pmatrix}, &
    X_4(Z) = \displaystyle \frac{1}{\sqrt{y_1 y_2}} \begin{pmatrix}
        \Im(\overline{z_1}z_2) & \Im(\overline{z_1}) \\ \Im(z_2) & 0
    \end{pmatrix}
\end{array}\] of $V(\BR)$, parametrized by $\BH^2$.

To draw the connection with the Grassmannian $\Gr(V)$ of positive definite planes of $V(\BR)$, we define the quantities 
\begin{align*}
    M(Z) &= \sqrt{y_1 y_2} (X_1(Z) + i X_2(Z)) = \begin{pmatrix}
        z_1 z_2 & z_1 \\ z_2 & 1
    \end{pmatrix}, \\
    M(Z)^\perp &= \sqrt{y_1 y_2} (X_3(Z) + i X_4(Z)) = \begin{pmatrix}
        \overline{z_1}z_2 & \overline{z_1} \\ z_2 & 1
    \end{pmatrix}.
\end{align*}
For $\gamma \in \SL_2(\BR)^2$, we have 
\[M(\gamma Z) = (c_1 z_1 + d_1)^{-1}(c_2 z_2 + d_2)^{-1}(\gamma \bullet M(Z)).\]
By abuse of notation, we denote by $M(Z)$, respectively $M(Z)^\perp$, the positive plane generated by the set $\{X_1(Z), X_2(Z)\}$, respectively the negative plane spanned by the pair $\{X_3(Z), X_4(Z)\}$.
The map $Z \mapsto M(Z)$ gives a bijection $\BH^2 \cong \Gr(V)$ that is compatible with the corresponding actions of $\SL_2(\BR)^2$.

\subsection{Hilbert modular forms as orthogonal modular forms} \label{sec:hilmfandorthmf}

We recall here the connection between orthogonal modular forms in signature $(2,2)$ and Hilbert modular forms. Let $\O^+(L)$ be the subgroup of $\SL_2(\BR^2)$ which maps $L$ to $L$, and let $L'/L$ be the discriminant group.
The latter is a finite abelian group on which the quadratic form $Q$ induces a $\BQ/\BZ$-valued quadratic form. We have a map $\O^+(L) \rightarrow O(L'/L)$
and we denote by $\Gamma(L)$ the subgroup of $\O^+(L)$ which acts trivially on $L'/L$. With $L$ as above, we have $\Gamma(L) = \SL_2(\CO_F)$.

Let $\Gamma \subset \SL_2(\BR)^2$ be a subgroup commensurable with $\SL_2(\CO_F)$. We say that a function $f: \BH^2 \rightarrow \BC$ is a Hilbert modular form
of weight $(k_1, k_2)$ for $\Gamma$ if it satisfies 
\[f(\gamma Z) = (c_1 z_1 + d_1)^{k_1}(c_2 z_2 + d_2)^{k_2}f(Z)\]
for all $\gamma = (\bpsm a_1 & b_1 \\ c_1 & d_1 \epsm, \bpsm a_2 & b_2 \\ c_2 & d_2 \epsm) \in \Gamma$ and $Z \in \BH^2$.
The previous paragraph allows us to identify Hilbert modular forms with modular forms for the group $\Gamma(L)$.

\subsection{Hilbert meromorphic forms} \label{sec:hilmf}
Zagier's original construction associated to a positive integer $m$ and an even integer $k \geq 4$ a Hilbert cusp form.
An analogous construction has been studied for vectors of negative norm, but before detailing it, let us define a quantity that will appear numerous time throughout the paper. For $Z = (z_1, z_2) \in \BH^2$ and $X = \bpsm a & \nu' \\ \nu &b \epsm \in V(\BR)$, let
\[q_Z(X) = (X, M(Z)) = -bz_1 z_2 + \nu z_1 + \nu' z_2 - a.\]
It obeys the following transformation law under the action of $\gamma \in \SL_2(\BR)^2$:
\[q_{\gamma Z}(X) = (c_1 z_1 + d_1)^{-1} (c_2 z_2 + d_2)^{-1} q_Z(\gamma^{-1}\bullet X).\]
Additionally, we denote by $X_{Z}$, respectively $X_{Z^\perp}$, the projection of $X$ on $M(Z)$, respectively on $M(Z)^\perp$.

Now, fix a vector $X \in V$ of negative norm, i.e. such that $Q(X)<0$, and consider the function 
\[\omega_{X}(Z) = \frac{(k-1)!}{2(2\pi)^k} \sum_{\gamma \in \Gamma_X \bs \Gamma} q_Z(\gamma^{-1} \bullet X).\]
It defines a Hilbert meromorphic modular form of weight $k$ for $\Gamma$ which vanishes at the cusps.
Fix $\beta \in L'/L$ and $m \in \frac{1}{D}\BZ_{<0}$. Since $\Gamma$ acts with finitely many orbits on the set 
\[L_{\beta, m} = \set{Y \in L+\beta: Q(Y) = m},\] 
we are able to define the average
\[\omega_{\beta, m}(Z) = \sum_{X \in \Gamma \bs L_{\beta, m}} \omega_{X}(Z) = \frac{(k-1)!}{2(2\pi)^k} \sum_{X \in L_{\beta,m}} q_Z(X)^{-k}.\]
Note in passing that by taking the sum $\sum_{\mu \in L'/L} \omega_{\mu, m}(Z)$, we recover the definitin given in the introduction.

\subsection{Hirzebruch-Zagier divisors and Weyl chambers}

We will want to study the Fourier expansions of the functions $\omega_m$ on a domain in which they are holomorphic. For $m < 0$, the poles of $\omega_m(z_1, z_2)$ are located at the following sets.

\begin{definition}
    Let $m < 0$. The \emph{Hirzebruch-Zagier divisor} $T_m$ of discriminant $m$ is defined as 
    \begin{align*}
        T_m & = \bigcup_{\substack{\lambda = \bpsm a & \nu' \\ \nu & b \epsm \in L'/ \set{\pm 1} \\ Q(\lambda) = m}} \set{Z = (z_1, z_2) \in \BH^2 : bz_1z_2 + \nu z_1 + \nu'z_2 + a = 0} \\ 
        & = \bigcup_{\substack{\lambda = \bpsm a & \nu' \\ \nu & b \epsm \in L'/\set{\pm 1} \\ Q(\lambda) = m}} \set{\left(\frac{-\nu' z_2 - b}{a z_2 + \nu}, z_2\right) \in \BH \times \BH}.
    \end{align*}
\end{definition}

Consider the set of pairs $(z_1, z_2) \in \BH^2$ such that their imaginary parts belong to
\[\BR^2_{>0} \setminus \bigcup_{\substack{\nu \in \invdiff{F} \\ N(\nu) = m}} \set{(y_1, y_2) \in \BR_{>0}^2 : \nu y_1 + \nu' y_2 = 0}.\]
This set isn't connected, but denote one of its connected components $C$, and let 
\[
    \CW = \set{Z = (z_1, z_2) \in \BH^2 : (\Im(z_1), \Im(z_2)) \in C}.
\]
Such sets are called \emph{Weyl chambers}.

\begin{theorem}[The Fourier expansion of $\omega_m$] \label{theo:fourierExpOmega}
    Let $m \in \frac{1}{D}\BZ_{< 0}$ and $z \in \CW$ with $\Im(z_1)\Im(z_2) > Dm$. The Fourier coefficient $c_{m\nu}$ of $\omega_m$ is given by
    \begin{align*}
        \left\{ \begin{array}{ll}
        \displaystyle \frac{2\pi}{\sqrt{D}} \left(\frac{N(\nu)}{\abs{Dm}}\right)^{(k-1)/2} \sum_{\alpha = 1}^\infty \frac{1}{\alpha} G_\alpha(Dm, \nu) I_{k-1}\left(\frac{4\pi}{\alpha} \sqrt{\abs{Dm}N(\nu)}\right), & \text{if } \nu \gg 0 \\
        \displaystyle (-1)^{k/2} \sum_{\substack{r \in \BN \\ r \mid \nu \sqrt{D} \\ \tr(\Im(z)\nu / r) > 0 \\ N(\nu/r) = Dm}} r^{k-1}, & \text{else.}
        \end{array} \right.
    \end{align*}
    Note that the condition $\tr(\Im(z)\nu/r)>0$ in the second sum only depends on the choice of a connected component $C$ in the definition of $\CW$ and not on $z$.
\end{theorem}

These were first computed by Zagier in the case $m>0$, and the analogous computations in the case $m<0$ were done in detail in \cite{roshardtMT}.

\section{The Doi-Naganuma theta lift} \label{sec:DoiNagLift}

\subsection{The theta lift}
A common way to obtain new modular forms out of known ones is via theta lifts. Here we consider a vector-valued version of the Doi-Naganuma lift that takes
harmonic weak Maass forms for the Weil representation attached to $L \cong \BZ \oplus \BZ \oplus \CO_F$ to modular forms for $\O(2,2)$, which is tantamount to lifting to Hilbert modular forms for $\Gamma_F = \SL_2(\CO_F)$ as explained (see \Cref{sec:hilmfandorthmf}).

\begin{definition}
    Let $\tau = u+iv \in \BH$. Using the notation from \Cref{sec:hilmf}, we define
    \begin{equation}
        \siegTheta{k}{L}(\tau, Z) = v \sum_{X \in L'} \left(\frac{q_Z(X)}{\Im(z_1)\Im(z_2)}\right)^k e(Q(X_Z)\tau + Q(X_{Z^\perp})\overline{\tau}) \Fe_{X+L}.
    \end{equation}
    Here $e(w) = e^{2\pi i w}$, as usual. We will sometimes use the notation $\Fe_\gamma(\tau) = e(\tau) \Fe_\gamma$.
\end{definition}
This theta function transforms like a modular form of weight $k$ for the Weil representation $\rho_L$ with respect to the variable $\tau$ and its complex conjugate transforms  like a modular form of parallel weight $k$ for $\Gamma_F$ in the variable $Z$.

For a harmonic Maass form $f \in H_{k, \rho_L}$, we let the theta lift of $f$ be given by 
\begin{equation}
    \Phi(f, Z) = \regint_{\CF} \left\langle f(\tau), \siegTheta{k}{L}(\tau, Z) \right\rangle v^k \frac{dudv}{v^2},
\end{equation}
where $\CF$ denotes a fundamental domain for $\SL_2(\BZ)$.

\subsection{Unfolding the lift against Maass-Poincaré series}

In this section, we compute the lift of the Maass-Poincaré series $F_{\beta, m}$.

\begin{proposition} \label{prop:unfoldingPoincSeries}
    For $m\in \frac{1}{D}\BZ$ and $\beta \in L'/L$ with $Q(\beta) = m$, we have
    \begin{align}
        \Phi(F_{\beta, m}(\tau, k/2),Z) = \frac{2^{k+2}}{(k-1)!}\omega_{\beta, m}(Z).
    \end{align}
\end{proposition}

\begin{proof}
    To simplify the notation, let us write $\Phi_{\beta,m}(Z,s)$ for the lift $\Phi(F_{\beta, m}(\tau, s), Z)$.
    By the usual unfolding argument, we obtain 
    \[\Phi_{\beta,m}(Z,s) = 2 \int_{v=0}^\infty \int_{u=0}^1 \CM_{m,k}(v,s) e(mu) \overline{\siegtheta{k}{L, \beta}(\tau, Z)} d\mu(\tau).\]
    Let us plug in the Fourier expansion
    \[\siegtheta{k}{L, \beta}(\tau, Z) = v \sum_{X \in L+\beta} \left(\frac{q_Z(X)}{\Im(z_1)\Im(z_2)}\right)^k \exp\left(-2\pi v \left(Q(X_Z) - Q(X_{Z^\perp})\right)\right) e\left(Q(X)u\right)\]
    and integrate over $u$. It yields,
    \begin{align*}
        \Phi_{\beta, m}(Z,s) = &\frac{2(4 \pi \abs{m})^{-k/2}}{\Gamma(2s)} \sum_{\substack{X \in L + \beta \\ Q(X) = m}} \left(\frac{\overline{q_Z(X)}}{\Im(z_1)\Im(z_2)}\right)^k \\
        &\times \int_{v=0}^\infty v^{\frac{k}{2}-1} M_{-k/2, s-1/2}(4\pi \abs{m}v) \exp(2\pi v(2 Q(X_{Z^\perp}) - m)) dv.
    \end{align*}
    As Bruinier noticed in~\cite{bruinier2002}, the last integral is a Laplace transform. We use \cite{tableint} page 215 (11) and get
    \begin{align*}
        \Phi_{\beta, m}(Z,s) = & \ \frac{2 (4 \pi \abs{m})^{s-k/2}}{\Gamma(2s)} \sum_{\substack{X \in L + \beta \\ Q(X) = m}} \left(\frac{\overline{q_Z(X)}}{\Im(z_1)\Im(z_2)}\right)^k \\ 
        & \times \Gamma\left(s+\frac{k}{2}\right)(-4\pi Q(X_{Z^\perp}))^{-s-\frac{k}{2}} \leftindex_2{F}_1\left[s+\frac{k}{2}, s+\frac{k}{2}, 2s, \frac{m}{Q(X_{Z^\perp})}\right].
    \end{align*}
    At the harmonic point $s = k/2$, we have 
    \[\leftindex_2{F}_1 \left[k,k,k, \frac{m}{Q(X_{Z^\perp})} \right] = \left(1-\frac{m}{Q(X_{Z^\perp})}\right)^{-k}.\]
    So, in total, we arrive to
    \begin{align*}
        \Phi_{\beta, m}(Z, k/2) & = 2 \sum_{\substack{X \in L + \beta \\ Q(X) = m}} \left(\frac{\overline{q_Z(X)}}{\Im(z_1)\Im(z_2)}\right)^k
        (-4\pi Q(X_{Z^\perp}))^{-k} \left(1-\frac{m}{Q(X_{Z^\perp})}\right)^{-k} \\
        & = 2\sum_{\substack{X \in L+\beta \\ Q(X) = m}} \left( \frac{\overline{q_Z(X)}}{4\pi\Im(z_1)\Im(z_2)Q(X_{Z})} \right)^k.
     \end{align*}
    Since
    \begin{align*}
        Q(X_Z) & = \frac{1}{4\Im(z_1)\Im(z_2)} \abs{(X,M(Z))}^2 = \frac{1}{4\Im(z_1)\Im(z_2)} \abs{q_Z(X)}^2,
    \end{align*}
    we conclude that
    \[\Phi_{\beta, m}(Z,k/2) = 2\cdot \pi^{-k}\sum_{\substack{X \in L + \beta \\ Q(X) = m}} q_Z(X)^{-k} = \frac{2^{k+2}}{(k-1)!}\omega_{\beta, m}(Z).\]
    This finishes the proof.
\end{proof}

\section{Divisibility of  the coefficients of Hilbert modular forms} \label{sec:IntAndDiv}

\subsection{The Fourier coefficients of the lift} 

We now use a result of Borcherds (Theorem 14.3 in \cite{borcherds1998}) to obtain an explicit formula for the Fourier coefficients of the lift.

\begin{theorem}
    Let $K$ be the lattice $\set{\bpsm 0 & \nu' \\ \nu & 0 \epsm : \nu \in \CO_F}$ so that $K' \cong \invdiff{F}$.
    Let 
    \[F(\tau) = \sum_{\gamma \in L'/L} \sum_{\substack{n \in \BZ + Q(\gamma) \\ n \gg -\infty}} c_\gamma(n) \Fe_\gamma(n\tau) \in M^!_{k, \rho_L}.\]
    In the Weyl chamber $\CW$, we have 
    \begin{align}
      C \cdot \Phi(F,Z) & = -c_0(0)\frac{B_k}{k} + 2\sum_{n > 0} \sum_{\substack{\lambda(\nu) = \bpsm 0 & \nu' \\ \nu & 0 \epsm \in K' \\ (\lambda, \CW) > 0}} e(n(\nu z_1 + \nu' z_2)) n^{k-1} c_{\lambda(\nu)}(N(\nu)) \nonumber \\
      & = -c_0(0) \frac{B_k}{k} + 2\sum_{\substack{\nu \in \invdiff{F} \\ (\nu, \CW) > 0}} \left[ \sum_{\substack{d \in \BN \\ d \mid \nu}} d^{k-1} c_{\lambda(\nu/d)}\left(\frac{N(\nu)}{d^2}\right) \right] e(\tr(\nu Z)).
    \end{align} \label{eq:coeffsDoiNaganuma}
    for $Z = \bpsm 0 & z_1 \\ z_2 & 0 \epsm \in W(\BC)$, where the condition $(\nu, \CW) > 0$ is a shorthand notation for $\nu y_1 + \nu' y_2 > 0$ for all $(z_1, z_2) \in \CW$ and $C = \sqrt{2}^{2-k}i^k$.
\end{theorem}

The second way of writing the expansion of $\Phi(F,Z)$ will come in handy to study its integrality. In the following, we just write $c_\nu$ instead of $c_{\lambda(\nu)}$ for $\nu \in \invdiff{F}$.

\subsection{Integrality and divisibility} 

Analysing \eqref{eq:coeffsDoiNaganuma} leads to the following observations:
\begin{itemize}
  \item If the coefficients $c_{\nu/d}\left(\frac{N(\nu)}{d^2}\right)$ of $F$ are rational for all $\nu \in \invdiff{F}$ and $d \in \BN$ such that $(\nu, \CW) > 0$ and $d \mid \nu$, then the coefficients of the lift of $F$ are rational.
  \item If the denominators of the said coefficients divide $d^{k-1}$, then the coefficients of the lift are integral.
\end{itemize}

The result that we obtain falls in between these in the following sense: we'll show that for each Maass-Poincaré series $F_{\beta, m}$ there exists an integer $\delta_{\beta,m}$ such that the coefficients of the lift $\Phi(F_{\beta, m}(\tau, k/2), Z)$ become integral when multiplied by $\delta_{\beta,m}$.
This property will be referred to as \emph{``having bounded denominators''}.
In order to prove this we will investigate $F_{\beta, m}$, the inverse image of $\omega_{\beta, m}$ via our theta lift. We'll use the following result of McGraw \cite{McGraw2003}:

\begin{theorem} \label{thm:corMcGraw}
  The space $M_{k, \rho_L}^!$ of weakly holomorphic modular forms of weight $k$ for the Weil representation attached to $L$ has a basis of modular forms all of whose have Fourier coefficients with bounded denominators.
\end{theorem}

Moreover, it follows easily from the next proposition, due to Bruinier-Funke \cite{bruinierfunke2004}, that when the space of cusps forms $S_{k, \overline{\rho}_L}$ is trivial, every harmonic Maass form of dual weight $2-k$ for $\rho_L$ is weakly holomorphic.

\begin{theorem} \label{thm:pairingWeaklyHolom}
    Let $F \in H_{2-k, \rho_L}$. Then $F$ is weakly holomorphic if and only if the pairing
    \[\{F,g\} = \sum_{\gamma \in L'/L}\sum_{n > 0} a_g(\gamma, n)a_F(\gamma, -n)\]
    vanishes for all cusp forms $g \in S_{k,\overline{\rho}_L}$.
\end{theorem}

We make the assumption $S_{k, \overline{\rho}_L} = \{0\}$. It follows from \Cref{thm:pairingWeaklyHolom} and \Cref{thm:corMcGraw} that the coefficients of $F_{\beta, m}$ have bounded denominators. Let us denote by $\delta_{\beta, m}$ the smallest integer that clears their denominators.
It follows that the Fourier coefficients of the Doi-Naganuma lift of $F_{\beta, m}$ given by \Cref{eq:coeffsDoiNaganuma} have rational coefficients whose denominators are cleared by $\delta_{\beta, m}$ as well.
To finish the proof of \Cref{thm:maintheorem}, we want to show that the coefficient of index $(\nu/d, N(\nu)/d^2)$ of the input of the lift is divisible by $(N(\nu)/d^2)^{k-1}.$ 

Fix $2 \leq k \in \BZ$ and consider $f \in H_{2-k, \rho_L}$. It follows from \cite{bringmannono2017} Theorem 5.5 that $\CD^{k-1}(f)$ belongs to $M^!_{k, \rho_L}$ with Fourier expansion
\begin{equation} \label{eq:FourierCoeffsBol}
    \CD^{k-1}f(\tau) \ = \ \sum_{\gamma \in L'/L} \sum_{\substack{n \in \BZ + Q(\gamma) \\ n \gg -\infty}} c_\gamma^+(n)n^{k-1} \Fe_\gamma(n\tau),
\end{equation}
where the $c^+_\gamma(n)$ denote the coefficients of the holomorphic part of the component $f_\gamma \Fe_\gamma$ of $f$.
We derive the following result from this observations.

\begin{proposition}\label{prop:implicationIntegrality}
    Let $g \in \CD^{k-1}(H_{2-k, \rho_L}) \subset M_{k, \rho_L}^!$ and assume that $S_{k, \overline{\rho}_L} = \{0\}$.
        Then, up to muliplication by $D^{k-1}$ and by an integer that doesn't depend on $\gamma \in L'/L$ nor on $n \in \BZ + Q(\gamma)$, the coefficient of index $(n,\gamma)$ of $g$
        is divisible by $(Dn)^{k-1}$.
\end{proposition}

\begin{proof}
    The condition $S_{k, \overline{\rho}_L} = \set{0}$ implies that $H_{2-k, \rho_L} = M^!_{2-k, \rho_L}$. 
    Hence, the preimage of $g$ via $\CD^{k-1}$ has bounded denominators. Denote $\delta_g$ the smallest integer clearing them.
    The fact that $\gamma \in L'/L$ and $n \in \BZ + Q(\gamma)$ implies that $n \in \frac{1}{D}\BZ$. 
    Hence, the coefficient of index $(\gamma,n)$ of $\delta_g D^{k-1}g$ is integral and divisible by $(Dn)^{k-1}$.
\end{proof}

\begin{proof}[Proof of Theorem \Cref{thm:maintheorem}] \label{prf:maintheorem1}
    The previous corollary brings us close but doesn't conclude the proof. We analyse the image of $F_{\beta, m, 2-k}$ (the added subscript specifies the weight of the Maass-Poincaré series) under the Bol operator more finely. 
    To simplify the notation, we write $F_{\beta, m,k}(\tau)$ for the evaluation of $F_{\beta, m, k}(\tau, s)$ at the harmonic point $s = 1-k/2$.
    By \Cref{eq:FourierCoeffsBol}, we have on one hand
    \begin{align}
    & \CD^{k-1}(F_{\beta, m, 2-k})(\tau) \nonumber \\ & = m^{k-1}e(m\tau)(\Fe_\beta + \Fe_{-\beta}) + \sum_{\gamma \in L'/L}\sum_{\substack{n \in \BZ + Q(\gamma) \\ n \geq 0}} n^{k-1} b_\gamma(n) \Fe_\gamma(n\tau),
    \end{align}
    where $b_\gamma(n)$ denotes the coefficient of index $(\gamma, n)$ of $F_{\beta, m, 2-k}$.
    On the other hand,
    \[\CD^{k-1}(F_{\beta,m,2-k}) = m^{k-1} F_{\beta,m,k}\]
    as computed in \cite{bringmannono2017} Theorem 6.11. Set $\delta_{\beta,m}$ to be the smallest integer that clears the denominators of the coefficients of $F_{\beta, m, 2-k}$.
    Now, the function
    \[ \left( D^{k-1}\delta_{\beta, m} \right) \CD^{k-1}(F_{\beta,m, 2-k}) = (mD)^{k-1}\delta_{\beta, m} F_{\beta, m, k} \] 
    has integral coefficients and its coefficient of index $(\gamma, n)$ is divisible by the integer $(Dn)^{k-1}$.

    So, using equation \Cref{eq:coeffsDoiNaganuma}, we deduce that the Fourier coefficient of index $\nu \in \invdiff{F}$ of the lift $\Psi(Z,(mD)^{k-1}\delta_{\beta, m} F_{\beta, m, k})$ can be written as
    \[\sum_{\substack{d \in \BN \\ d \mid \nu}} \left(\frac{DN(\nu)}{d}\right)^{k-1} \tilde{c}_{\nu/d}\left(\frac{N(\nu)}{d^2}\right), \]
    for an integer $\tilde{c}_{\nu/d}\left(\frac{N(\nu)}{d^2}\right)$.

    Finally, for $\nu \in \invdiff{F}$, let $\nu_0 \in \invdiff{F}$ such that $\BQ \nu \cap L' = \BZ\nu_0$ and denote $\ell_\nu$ the unique integer such that 
    \[\nu = \ell_\nu \nu_0.\]
    We can now rewrite the last sum as
    \[(D\ell_\nu N(\nu_0))^{k-1} \sum_{\substack{d \in \BN \\ d \mid \ell_\nu}} \left(\frac{\ell_\nu}{d}\right)^{k-1} \tilde{c}_{\frac{\ell_\nu}{d}\nu_0}\left(N\left(\frac{\ell_\nu}{d}\nu_0\right)\right).\]
    This concludes the proof that, up to multiplication by an integer that doesn't depend on $\nu$, the $N(\nu)$-th coefficien of $\omega_m$ is divisible by $(D\ell_\nu N(\nu_0))^{k-1}$.
\end{proof}

\subsection{Non-trivial space of cusp forms}

We proved \Cref{thm:maintheorem} under the assumption that the space of cusp forms $S_{k, \overline{\rho}_L}$ was trivial. We now investigate the non-trivial case.
If we remove this condition, the preimage $F_{\beta, m, 2-k}$ of the Mass-Poincaré series $F_{\beta,m,k}$ via $\CD^{k-1}$ is a form whose coefficient might be irrational. 
Hence, the coefficients of its image via the Doi-Naganuma lift are a priori irrational. However, we are able to find linear combinations of the $\omega_{\beta,m}$'s for which the desired property holds.

\begin{theorem} \label{theo:mainTheorem}
    Let $k > 2$ be an even integer and let $G \in M_{2-k, \rho_L}^!$ with Fourier expansion $G(\tau) = \sum_{\gamma \in L'/L}\sum_{\substack{n \in \BZ + Q(\gamma) \\ -\infty \ll n}} c_G(\gamma, n) e(n\tau)$.
    Then, the coefficient of order $\nu \in \invdiff{F}$ of the meromorphic Hilbert modular form
    \begin{equation} \label{eq:defOmega}
        \sum_{\substack{\beta \in L'/L \\ n < 0}} c_G(\beta,n) \omega_{\beta, n}(Z)
    \end{equation}
    has bounded denominators and is divisible by $(D\ell_\nu N(\nu_0))^{k-1}$.
\end{theorem}

\begin{proof}
    Let $\Omega_{G}(Z)$ be the function defined in \Cref{eq:defOmega}.
    It is the image of 
    \[\CF_{G}(\tau) = \frac{(k-1)!}{2^{k+2}} \sum_{\substack{\beta \in L'/L \\ n < 0}} \frac{c_G(\beta, n)}{n^{k-1}} F_{\beta, n, 2-k}(\tau, k/2)\]
    via the composition of the Bol operator followed by the Doi-Naganuma lift.
    Since $G \in M_{2-k, \rho_L}^!$, the pairing from \Cref{thm:pairingWeaklyHolom} of $\CF_G(\tau)$ against any cusp form in $S_{k, \overline{\rho}_L}$ vanishes. 
    This implies in turn that $\CF_G(\tau) \in M_{2-k, \rho_L}^!$. 
    Now, the theorem follows via a similar reasoning as in the proof of \Cref{thm:maintheorem}.
\end{proof}

\printbibliography

\end{document}